\documentclass{amsart}
\usepackage{amssymb,latexsym}
\usepackage{amscd,amsthm}

\usepackage[all]{xy}
\usepackage{tikz}

\usepackage{tikz-cd}

\newtheorem{theorem}{Theorem}[section]

\newtheorem{lemma}[theorem]{Lemma}
\newtheorem{proposition}[theorem]{Proposition}
\newtheorem{corollary}[theorem]{Corollary}

\theoremstyle{definition}
\newtheorem{definition}[theorem]{Definition}

\newtheorem{remark}[theorem]{Remark}

\newtheorem{example}[theorem]{Example}

\DeclareMathOperator{\Ext}{Ext}
\DeclareMathOperator{\Hom}{Hom}

\newcommand{\cat}[1]{\mathcal{#1}}           

\newcommand{\tensor}{\otimes}

\newcommand{\class}[1]{\mathcal{#1}}   

\newcommand{\Z}{\mathbb{Z}}

\newcommand{\ch}{\textnormal{Ch}(R)}

\newcommand{\cha}[1]{\textnormal{Ch}(\mathcal{#1})}

\newcommand{\tilclass}[1]{\widetilde{\class{#1}}}
\newcommand{\dgclass}[1]{dg\widetilde{\class{#1}}}
\newcommand{\dwclass}[1]{dw\widetilde{\class{#1}}}
\newcommand{\exclass}[1]{ex\widetilde{\class{#1}}}

\newcommand{\rightperp}[1]{#1^{\perp}}
\newcommand{\leftperp}[1]{{}^\perp #1}

\newcommand{\homcomplex}{\mathit{Hom}}

\begin{document}

\title[K-flat complexes in Grothendieck categories]{K-flatness in Grothendieck categories: Application to quasi-coherent sheaves}

\author{Sergio Estrada}
 \address{S.E. \ Universidad de
  Murcia\\ Departamento de Matemáticas \\ Campus de Espinardo \\ Murcia 30100, Spain}
\email[Sergio Estrada]{sestrada@um.es}
\urladdr{}

\author{James Gillespie}
\address{J.G. \ Ramapo College of New Jersey \\
         School of Theoretical and Applied Science \\
         505 Ramapo Valley Road \\
         Mahwah, NJ 07430\\ U.S.A.}
\email[Jim Gillespie]{jgillesp@ramapo.edu}
\urladdr{http://pages.ramapo.edu/~jgillesp/}

\author{Sinem Odabaşı}
 \address{S.O. \ Universidad de
  Murcia\\ Departamento de Matemáticas \\ Campus de Espinardo \\ Murcia 30100, Spain}
\email[Sinem Odabaşı]{sinem.odabasi@um.es}
\urladdr{}

\date{\today}



\thanks{The first and the third named authors were partly supported by grant
  PID2020-113206GB-I00 funded by MCIN/AEI/10.13039/ 501100011033 and by grant 22004/PI/22 funded by Fundaci\'on S\'eneca.}

\begin{abstract}
Let $(\cat{G},\otimes)$ be any closed symmetric monoidal Grothendieck category. We show that K-flat covers exist universally in the category of chain complexes and that the Verdier quotient of $K(\cat{G})$ by the K-flat complexes is always a well generated triangulated category. Under the further assumption that $\cat{G}$ has a set of $\otimes$-flat generators we can show more: (i) The category is in recollement with the $\otimes$-pure derived category and the usual derived category, and (ii) The usual derived category is the homotopy category of a cofibrantly generated and monoidal model structure whose cofibrant objects are precisely the K-flat complexes. We also give a condition guaranteeing that the right orthogonal to K-flat is precisely the acyclic complexes of $\otimes$-pure injectives. We show this condition holds for quasi-coherent sheaves over a quasi-compact and semiseparated scheme.
\end{abstract}

\maketitle

\section{Introduction}\label{sec-intro}

In~\cite{estrada-gillespie-odabasi} we defined and studied the pure derived category of a quasi-compact and semiseparated scheme. An important point is that there are two natural  notions of purity in the category of quasi-coherent sheaves over such a scheme. First, there is the categorical purity arising from the fact that this category is locally finitely presented. The second notion is obtained by considering the short exact sequences that remain exact upon tensoring with any quasi-coherent sheaf. These two notions of purity coincide for affine schemes, but typically they differ. In general, the more natural notion of purity for quasi-coherent sheaves is the second one, for it is equivalent to having purity on the stalks and this notion is in agreement with the local nature of the tensor product and flatness for sheaves. Hence we call this the \emph{geometric} purity while the former is the \emph{categorical} purity. Motivated by studying the geometric pure derived category of a scheme, we introduced in~\cite{estrada-gillespie-odabasi} the $\otimes$-pure derived category of a general closed symmetric monoidal Grothendieck category $(\cat{G},\otimes)$. Recall that a Grothendieck category is a cocomplete abelian category $\cat{G}$ with a set of generators and such that direct limits of monomorphisms are again monomorphisms. In practice, the category $\cat{G}$ often possesses a tensor product $\otimes$ providing the formal structure of a closed symmetric monoidal category.

Now given any closed symmetric monoidal Grothendieck category $(\cat{G},\otimes)$, all of this structure  lifts to chain complexes, resulting in a closed symmetric monoidal structure $(\cha{G},\otimes)$ on the category of complexes. Spaltenstein's notion of a K-flat complex from~\cite{spaltenstein} makes sense in this general setting. A chain complex $X$ is said to be K-flat if $E \otimes X$ is an acyclic (i.e. exact) chain complex whenever $E$ is acyclic. In other words, the functor $- \otimes X$ preserves acyclicity.
Given a ring $R$, Emmanouil studied in~\cite{emmanouil-K-flatness-and-orthogonality-in-homotopy-cats} the Verdier quotient of the chain homotopy category, $K(R)$, by the thick class of all K-flat complexes. He shows it to be equivalent to $K_{ac}(\class{PI})$, the chain homotopy category of all acyclic complexes of pure-injective $R$-modules. In this paper we prove an analog of this for quasi-coherent sheaves, using the geometric pure injective (i.e. $\otimes$-pure injective) quasi-coherent sheaves.

But first we prove some results of interest that hold in the full generality of any  closed symmetric monoidal Grothendieck category $(\cat{G},\otimes)$. Let $\class{KF}$ denote the class of all K-flat chain complexes. First, Corollary~\ref{corollary-K-flat-covers} shows that $\class{KF}$ is a covering class in the category $\cha{G}$ of chain complexes. In particular, it is a special precovering class: For any complex $Y$ there is a K-flat complex $X$ and a surjective chain map $p : X \xrightarrow{} Y$ such that any other chain map $f : X' \xrightarrow{} Y$,  with $X'$ K-flat, lifts over $p$.
The second result concerns the Verdier quotient of $K(\cat{G})$, the chain homotopy category of $\cat{G}$, by the class $\class{KF}$. Theorem~\ref{theorem-K-derived-category} implies that the Verdier quotient, $K(\cat{G})/\class{KF}$, is always at least a well-generated (if not compactly generated) triangulated category. In particular, the above results hold for any scheme $\mathbb{X}$. That is, any chain complex of quasi-coherent sheaves has a K-flat cover, and $K(\mathbb{X})/\class{KF}$ is always well generated.

Theorem~\ref{theorem-K-derived-category} actually constructs a cofibrantly generated  abelian model structure on $\cha{G}$ whose weak equivalences are precisely the chain maps whose mapping cone is a  K-flat complex. Every complex is cofibrant in this model structure and the fibrant objects are certain complexes with $\otimes$-pure injective components. This model was constructed on chain complexes of modules over a ring $R$ in \cite{gillespie-K-flat}. But constructing the model structure in the present generality requires us to expand upon techniques from~\cite{estrada-gillespie-odabasi} which, in particular, utilize the theory of $\lambda$-purity from~\cite{AR}.

In Section~\ref{sec-flat-generators}, we are able to say more under the assumption that  $\cat{G}$ possesses a generating set consisting of $\otimes$-flat objects. An object $F \in \cat{G}$ is $\otimes$-flat if $F\otimes -$ is an exact functor. With the presence of $\otimes$-flat generators for $\cat{G}$, we display a recollement linking $K(\cat{G})/\class{KF}$ to the usual derived category, $\class{D}(\cat{G})$, and the $\otimes$-pure derived category, $\class{D}_{\otimes\textnormal{-pur}}(\cat{G})$, which was introduced in~\cite{estrada-gillespie-odabasi}.  See Theorem~\ref{theorem-recollement}.
Related to this, we see in Theorem~\ref{theorem-K-flat-model-derived-cat} the existence of a monoidal model structure whose homotopy category is $\class{D}(\cat{G})$. This model is equivalent to the usual flat model structure but in some sense it is more natural because every K-flat complex is cofibrant, not just the ones with $\otimes$-flat components.

Finally, Section~\ref{sec-right-orthogonal} focuses on the central example of quasi-coherent sheaves over a quasi-compact semiseparated scheme $\mathbb{X}$. With the help of a technical but useful result concerning cotorsion pairs and direct limits (Theorem~\ref{theorem-direct-limits}) we imitate the {\v{C}}ech resolution  argument from~\cite[Theorem 3.3]{cet-G-flat-stable-scheme}, to obtain the wanted analog of Emmanouil's result concerning the Verdier quotient $K(R)/\class{KF}$: The right Ext-orthogonal to $\class{KF}$ is precisely the class of all acyclic complexes of geometric pure injective ($\otimes$-pure injective) quasi-coherent sheaves. It follows that the Verdier quotient $K(\mathbb{X})/\class{KF}$ is equivalent to $K_{ac}(\class{PI}_{\otimes})$, the chain homotopy category of all acyclic complexes of $\otimes$-pure injectives; see  Theorem~\ref{theorem-qc-sheaves}. We also consider the affine case $\mathbb{X} = \text{Spec}(R)$. Here, taking advantage of having not only enough flats but utilizing that we also  have enough pure-projectives and that geometric and categorical purity agree in this case, we see that there is another  abelian model structure for $K(R)/\class{KF}$. Its  cofibrant objects are the chain complexes of pure projective modules. The resulting homotopy category reflects an  equivalence of $K(R)/\class{KF}$ with the Verdier quotient of $K(R)$ by Emmanouil and Kaperonis' \emph{K-absolutely pure} complexes, introduced in~\cite{emmanouil-kaperonis-K-flatness-pure}.

\section{Purity in symmetric monoidal Grothendieck categories}\label{sec-pure exact struc}

Throughout this section we let $(\cat{G}, \otimes)$ be a closed symmetric monoidal Gro\-then\-dieck category. Recall that $\cat{G}$ is locally $\lambda$-presentable (in the sense of~\cite{AR}) for some regular cardinal $\lambda$, and, as described in~\cite[Section~2]{estrada-gillespie-odabasi}, there are two generally different notions of purity. Let $\mathcal{P}_{\lambda}$ denote the proper class of $\lambda$-pure short exact sequences in $\cat{G}$, and $\mathcal{P}_{\otimes}$ denote the proper class of $\otimes$-pure short exact sequences in $\cat{G}$. We have the containment $\mathcal{P}_{\lambda} \subseteq \mathcal{P}_{\otimes}$; see~\cite[Remark~2.8]{estrada-gillespie-odabasi}, and the book~\cite{AR} is the standard reference for locally $\lambda$-presentable categories and $\lambda$-purity.

We will denote by $\cha{G}_{\otimes}$ the exact structure consisting of $\cha{G}$, the category of $\cat{G}$-chain complexes, along with the componentwise $\tensor$-pure exact sequences. The associated Yoneda Ext group of all (equivalence classes of) degreewise $\tensor$-pure short exact sequences will be denoted by $\Ext^1_{\otimes}(X,Y)$.   Note that since $X \otimes -$ preserves direct limits for any object $X\in\cat{G}$, any transfinite composition of $\otimes$-pure monomorphisms is again a $\otimes$-pure monomorphism.

 The following result tells us that we can use $\Ext^1_{\otimes}$ to construct complete cotorsion pairs in $\cha{G}_{\otimes}$ by the usual method of cogenerating by a set.  To state it, let $\lambda$ be some regular cardinal for which $\cat{G}$ is locally $\lambda$-presentable. Let $\{L_i\}_{i\in \Lambda}$ be a set of representatives of all the isomorphism classes of $\lambda$-presented objects.  This is a generating set for the exact category $(\cat{G}, \mathcal{P}_{\lambda})$, meaning every object is a $\lambda$-pure epimorphic image of some direct sum of copies of some $L_i$. Since any $\lambda$-pure epic is also a $\otimes$-pure epic we see that $\{L_i\}$ is also a generating set for the exact structure $(\cat{G}, \mathcal{P}_{\otimes})$ \cite[Prop. 2.9]{estrada-gillespie-odabasi}. By taking all their $n$-disks ($n \in \Z$), this lifts to a set $\{D^n(L_i)\}$ of generators for the exact category $\cha{G}_\otimes$.

\begin{lemma}\label{cogen-by-set-complete}
Let $\class{S}$ be a set (not a proper class) of chain complexes such that $\leftperp{(\rightperp{\class{S}})}$ contains the generating set $\{D^n(L_i)\}$ for the exact category $\cha{G}_\otimes$. Then $(\leftperp{(\rightperp{\class{S}})}, \rightperp{\class{S}})$ is a complete cotorsion pair in the exact category $\cha{G}_\otimes$.  Moreover, the class $\leftperp{(\rightperp{\class{S}})}$ consists precisely of direct summands (retracts) of transfinite degreewise $\tensor$-pure extensions of complexes in $\class{S}\cup \{D^n(L_i)\}$.
\end{lemma}

\begin{proof}
This follows from~\cite[Corollary~2.15]{saorin-stovicek} since $(\cat{G}, \mathcal{P}_{\otimes})$, and hence $\cha{G}_\otimes$, is an efficient exact category, by~\cite[Lemma 3.6]{estrada-gillespie-odabasi}.
\end{proof}

Let us recall how $\otimes$ lifts to a closed symmetric monoidal structure on $\cha{G}$. Given $X, Y \in \cha{G}$, their tensor product $X
\otimes Y$ is defined by $$(X \otimes Y)_n = \bigoplus_{i+j=n} (X_i
\otimes Y_j)$$ in degree $n$. The boundary map $\delta_n$ is induced
by the formula $d^X_i \otimes 1_{Y_j} +
(-1_{X_i})^{i} \otimes d^Y_j$. The closed structure also lifts from $\cat{G}$ to $\cha{G}$, but we won't need its description in this paper. The point is that the above tensor product gives us a  symmetric monoidal structure on $\cha{G}$ with $X \otimes -$ a left adjoint for each $X \in \cha{G}$.  If $I$ denotes the unit of the monoidal structure on $\cat{G}$, then $S^0(I)$, the trivial complex with $I$ concentrated in degree 0,  becomes the unit for the monoidal structure on $\cha{G}$.

The following lemma states that the $\otimes$-pure short exact sequences of chain complexes are precisely the short exact sequences in $\cha{G}_{\otimes}$.

\begin{lemma}\label{lemma-degreewise-purity}
$0 \xrightarrow{} W \xrightarrow{} X \xrightarrow{} Y \xrightarrow{} 0$ is a short exact sequence in the exact category $\cha{G}_\otimes$ if and only if $0 \xrightarrow{} C \otimes W \xrightarrow{}C \otimes X \xrightarrow{} C \otimes Y \xrightarrow{} 0$ remains exact in $\cha{G}$ for every chain complex $C$.
\end{lemma}

\begin{proof}
($\implies$) Assume
$0 \xrightarrow{} W \xrightarrow{} X \xrightarrow{} Y \xrightarrow{} 0$ is $\otimes$-pure exact in each degree, and let $C \in \cha{G}$. Then for all pairs of integers $i, j$ we have short exact sequences in $\cat{G}$
$$0 \xrightarrow{} C_i \tensor W_j \xrightarrow{}C_i \tensor X_j \xrightarrow{} C_i \tensor Y_j \xrightarrow{} 0.$$
Since short exact sequences in $\cat{G}$ are closed under direct sums, it follows that for each $n \in \Z$ we have short exact sequences
$$0 \xrightarrow{} \bigoplus_{i+j=n}  C_i \tensor W_j \xrightarrow{} \bigoplus_{i+j=n} C_i \tensor X_j \xrightarrow{} \bigoplus_{i+j=n} C_i \tensor Y_j \xrightarrow{} 0.$$
By definition, this is the degree $n$ component of the tensor products, so we get a short exact sequence of chain complexes
$$0 \xrightarrow{} C \tensor W \xrightarrow{} C \tensor X \xrightarrow{} C \tensor Y \xrightarrow{} 0.$$

($\impliedby$) Use that  $S^0(M) \tensor X = M \otimes X$, for any object $M \in \cat{G}$.
\end{proof}

For a subcomplex $X \subseteq Y$ we will write $X \subseteq_{\otimes} Y$ to mean that each inclusion $X_n \subseteq Y_n$ is a $\otimes$-pure monomorphism. In other words, the inclusion $X \subseteq Y$ is an admissible monic in the exact category $\cha{G}_{\otimes}$.

\begin{lemma}\label{lemma-purity-prop}
Let $X  \subseteq Y \subseteq Z$ be subcomplexes in $\cha{G}$.
If $X \subseteq_{\otimes}Z$  and $Y/X \subseteq_{\otimes} Z/X$, then also $Y \subseteq_{\otimes} Z$.
\end{lemma}

\begin{proof}
We have a comutative diagram in $\cha{G}$ with exact rows and with each vertical arrow a monomorphism.
$$\begin{CD}
     0   @>>>    X    @>>>    Y    @>>>     Y/X    @>>>    0 \\
    @.        @|     @VVV      @VVV   @.\\
        0   @>>>     X   @>>>    Z    @>>>    Z/X   @>>>    0 \\
    \end{CD}$$
    We will use the characterization in Lemma~\ref{lemma-degreewise-purity}.
Since $X \subseteq_{\otimes} Z$, applying $W \otimes - $ for any complex $W$ yields another commutative diagram
$$\begin{CD}
     0   @>>>    W \otimes X    @>>>   W \otimes Y    @>>>    W \otimes Y/X    @>>>    0 \\
    @.        @|     @VVV      @VVV   @.\\
        0   @>>>    W \otimes X   @>>>   W \otimes Z    @>>>   W \otimes Z/X   @>>>    0 \\
    \end{CD}$$
with exact rows. Now since $Y/X \subseteq_{\otimes} Z/X$, the rightmost vertical morphism is monic. So by the snake lemma, the middle vertical morphism is monic which means $Y \subseteq_{\otimes}Z$.
\end{proof}

\section{$\otimes$-acyclicity of chain complexes}\label{section-C-acyclic}

Throughout this section we again let $(\cat{G}, \otimes)$ be a closed symmetric monoidal Grothendieck category and we let $\lambda$ denote a regular cardinal for which $\cat{G}$ is locally $\lambda$-presentable. We note that $\cha{G}$ is then locally $\lambda$-presentable too.

Many of the results we would like to prove for K-flat complexes are actually special cases of a more general phenomenon. In particular, K-flat complexes are $\class{C}$-acyclic in the following sense, by taking $\class{C}$ to be the class of all exact complexes.

\begin{definition}\label{def-C-acyclic}
Let $\class{C}$ be any given fixed class of chain complexes in $\cha{G}$. We will say that a chain complex $X$ is \emph{$\class{C}$-acyclic} if $C \tensor X$ is acyclic for all $C \in \class{C}$. We let ${}_\class{C}\class{W}$ denote the class of all $\class{C}$-acyclic complexes $X$.
\end{definition}

In addition to K-flat complexes, the pure acyclic complexes (see Section~\ref{sec-k-flat-pure-ac}) are also a special case. Note too that the $\{S^0(I)\}$-acyclic complexes are precisely the usual acyclic (exact)  complexes.

\begin{lemma}\label{lemma-lambda-pure-subs}
 Let $P\subseteq X$ be a $\lambda$-pure subcomplex of $X \in \cha{G}$. If $X$ is acyclic (just exact not necessarily pure acyclic), then both $P$ and $X/P$ are acyclic too. In fact, if $X$ is $\class{C}$-acyclic, then $P$ and $X/P$ are also $\class{C}$-acyclic.
\end{lemma}

 \begin{proof}
We generalize ideas inside the proof of~\cite[Prop.~3.11]{estrada-gillespie-odabasi}. In particular, if $L$ is any $\lambda$-presented object of $\cat{G}$, then $S^n(L)$ and $D^n(L)$ are $\lambda$-presented in $\cha{G}$. So by applying $\Hom_{\cha{G}}(S^n(L),-) \cong \Hom_{\cat{G}}(L,Z_n(-))$ and $\Hom_{\cha{G}}(D^n(L),-) \cong \Hom_{\cat{G}}(L,(-)_n)$ we find that the rows of the commutative diagram must be $\lambda$-pure exact sequences:
$$\begin{CD}
     0   @>>>    Z_nP    @>>>   Z_nX    @>>>    Z_n(X/P)    @>>>    0 \\
    @.        @VVV     @VVV      @VVV   @.\\
        0   @>>>    P_n   @>>>   X_n    @>>>   (X/P)_n    @>>>    0 \\
    \end{CD}$$
 In particular they are usual short exact sequences and so applying the snake lemma we deduce that $0 \xrightarrow{} B_{n-1}P \xrightarrow{} B_{n-1}X_n \xrightarrow{} B_{n-1}(X/P) \xrightarrow{} 0$ is a short exact sequence.
But then applying the snake lemma again to
 $$\begin{CD}
     0   @>>>    B_nP    @>>>   B_nX    @>>>    B_n(X/P)    @>>>    0 \\
    @.        @VVV     @VVV      @VVV   @.\\
     0   @>>>    Z_nP    @>>>   Z_nX    @>>>    Z_n(X/P)    @>>>    0 \\
    \end{CD}$$
    produces a short exact sequence in homology  $0 \xrightarrow{} H_nP \xrightarrow{} H_nX \xrightarrow{} H_n(X/P) \xrightarrow{} 0$. We conclude that $P$ and $X/P$ are each acyclic whenever $X$ is acylic.

We now show that if $X$ is $\class{C}$-acyclic, then $P$ and $X/P$ are also $\class{C}$-acyclic. Since $0 \xrightarrow{} P \xrightarrow{} X \xrightarrow{} X/P \xrightarrow{} 0$ is a $\lambda$-pure exact sequence it is characterized as a $\lambda$-directed colimit of splitting short exact sequences; see~\cite[2.30, page 86]{AR} and~\cite[Prop.~6.5]{gillespie-G-derived}. For any complex $C$ we have that $C \otimes -$ preserves direct limits, and since direct limits are exact it follows that $0 \xrightarrow{} C \otimes P \xrightarrow{} C \otimes X \xrightarrow{} C \otimes  X/P \xrightarrow{} 0$ is also a $\lambda$-directed colimit of splitting short exact sequences. Therefore it too is a $\lambda$-pure exact sequence. So whenever $C \otimes X$ is acyclic, it follows from  what we already proved above that both $C \otimes  P$ and $C \otimes  X/P$ are also acyclic. In particular, for any class of complexes $\class{C}$, both $P$ and $X/P$ are $\class{C}$-acyclic whenever $X$ is $\class{C}$-acyclic.
 \end{proof}

\begin{proposition}\label{prop-C-acyclic-properties}
Let $\class{C}$ be any class of chain complexes. The class ${}_\class{C}\class{W}$ of all $\class{C}$-acyclic complexes satisfies the following properties.
\begin{enumerate}
\item ${}_\class{C}\class{W}$ is closed under direct sums, direct summands (retracts), and direct limits. Also, ${}_\class{C}\class{W}$ contains all contractible complexes.
\item ${}_\class{C}\class{W}$ is thick in $\cha{G}_{\otimes}$. That is, it satisfies the 2 out of 3 property with respect to short exact sequences of chain complexes that are $\otimes$-pure exact in each degree.
\item ${}_\class{C}\class{W}$ is closed under transfinite extensions in $\cha{G}_{\otimes}$. That is, if $X$ has a filtration $X \cong \varinjlim_{\alpha < \beta} X_{\alpha}$ where each $X_{\alpha} \xrightarrow{} X_{\alpha +1}$ is a degreewise $\otimes$-pure monomorphism, and $X_0$ and each $X_{\alpha +1}/X_{\alpha}$ are $\class{C}$-acyclic, then $X$ too is $\class{C}$-acyclic.
\end{enumerate}
\end{proposition}

\begin{proof}
Since we are in a Grothendieck category the class of all acyclic complexes is closed under direct sums, direct summands, and direct limits. For any chain complex $C$, the functor $C \otimes -$ preserves  direct sums, direct summands, and direct limits. So ${}_\class{C}\class{W}$ is closed under these operations. As for the contractible complexes being $\class{C}$-acyclic, recall that any contractible complex $X$ must take the form $X \cong \bigoplus_{n\in\Z}D^n(M_n)$ for some $\cat{G}$-objects $\{M_n\}_{n\in\Z}$. So then we have  $C \otimes X \cong \bigoplus_{n\in\Z} (C\otimes D^n(M_n))$, and each $C\otimes D^n(M_n)$ can be shown to be acyclic. For example, see~\cite[Exercise~1.2.5]{weibel}.

Now let $C \in \class{C}$. By Lemma~\ref{lemma-degreewise-purity} we get a short exact sequence of chain complexes
$$0 \xrightarrow{} C \tensor W \xrightarrow{} C \tensor X \xrightarrow{} C \tensor Y \xrightarrow{} 0.$$
So then if 2 out of 3 of these are acyclic, so is the third. Therefore ${}_\class{C}\class{W}$ satisfies the 2 out of 3 property on short exact sequences in $\cha{G}_{\otimes}$.

So now it is easy to see that ${}_\class{C}\class{W}$ is closed under transfinite extensions in $\cha{G}_{\otimes}$. This follows from the fact that we have shown ${}_\class{C}\class{W}$ to be closed under extensions in $\cha{G}_{\otimes}$, and under direct limits.
\end{proof}

\begin{theorem}\label{C-acyclics-cotorsion-pair}
Let $\class{C}$ be any class of $\cat{G}$-chain complexes and ${}_\class{C}\class{W}$ the class of all $\class{C}$-acyclic complexes. Set $\class{F} := \rightperp{{}_\class{C}\class{W}}$, defined in the exact category $\cha{G}_{\otimes}$. That is,
 $$\class{F} = \{\,F \in \cha{G} \,|\, \Ext^1_{\otimes}(X,F)=0\  \forall \,X \in \,{}_\class{C}\class{W}\}.$$
Then $({}_\class{C}\class{W}, \class{F})$ is a complete cotorsion pair, cogenerated by a set, in the exact category $\cha{G}_{\otimes}$. It satisfies the following properties:
\begin{enumerate}
\item ${}_\class{C}\class{W}$ is thick in $\cha{G}_{\otimes}$. In particular the cotorsion pair is hereditary.
\item $F \in \class{F}$ if and only if each $F_n$ is $\otimes$-pure injective and every chain map $X \xrightarrow{} F$ is null homotopic whenever $X \in {}_\class{C}\class{W}$.
\item ${}_\class{C}\class{W}\cap \class{F}$ equals the class of all injective objects in the exact category $\cha{G}_{\otimes}$. These are precisely the contractible complexes with $\otimes$-pure-injective components.
\end{enumerate}
\end{theorem}

 \begin{proof}
Since $\cha{G}$ is locally $\lambda$-presentable there exists, up to isomorphism, only a set (as opposed to a proper class) of $\gamma$-presentable chain complexes for each regular cardinal $\gamma$ (\cite[Corollary~1.69]{AR}). Moreover, by~\cite[Theorem~2.33]{AR} there exist arbitrarily large regular cardinals $\gamma \rhd \lambda$ such that  every $\gamma$-presentable subcomplex $S \subseteq X$ in $\cha{G}$ is contained in a $\lambda$-pure subcomplex $P \subseteq X$, where $P$ is $\gamma$-presentable. So we choose such a $\gamma$ and let $\class{S}$  be a set of isomorphism representatives for all $\gamma$-presentable $\class{C}$-acyclic chain complexes. Lemma~\ref{cogen-by-set-complete} applies because $\{D^n(L_i)\} \subseteq \class{S}$. (Each $D^n(L_i)$ is $\class{C}$-acylic and $\gamma$-presentable, since $\gamma > \lambda$.) So we obtain a complete cotorsion pair $(\leftperp{(\rightperp{\class{S}})}, \rightperp{\class{S}})$ in the exact category $\cha{G}_\otimes$. Moreover, the class $\leftperp{(\rightperp{\class{S}})}$ consists precisely of direct summands (retracts) of transfinite degreewise $\tensor$-pure extensions of complexes in $\class{S} = \class{S}\cup \{D^n(L_i)\}$.

We will show  $\leftperp{(\rightperp{\class{S}})} = {}_\class{C}\class{W}$. Since $\class{S}\subseteq {}_\class{C}\class{W}$, the containment $\leftperp{(\rightperp{\class{S}})} \subseteq {}_\class{C}\class{W}$ follows from Proposition~\ref{prop-C-acyclic-properties}.  On the other hand, we will show ${}_\class{C}\class{W} \subseteq \leftperp{(\rightperp{\class{S}})}$ by showing that every $X \in {}_\class{C}\class{W}$ is a transfinite degreewise $\tensor$-pure extension of complexes in $\class{S}$.

But given $X\in
{}_\class{C}\class{W}$, we may use transfinite induction to construct a filtration of $X$ by $\class{C}$-acyclic subcomplexes $P_{\alpha}\subseteq_{\otimes} X$ with $P_{0}\in\class{S}$ and each $P_{\alpha+1}/P_{\alpha} \in \class{S}$. (Recall that $\subseteq_{\otimes}$ denotes a degreewise $\otimes$-pure subcomplex.)  For $i=0$,
we let $P_{0}$ be a nonzero $\gamma$-presentable and $\lambda$-pure subcomplex of $X$.  Being a $\lambda$-pure subcomplex of the $\class{C}$-acyclic $X$, we note that $P_0$ must also be $\class{C}$-acyclic, by Lemma~\ref{lemma-lambda-pure-subs}. Moreover, $P_0 \subseteq_{\otimes} X$, since $\lambda$-pure subcomplexes are degreewise $\lambda$-pure, hence degreewise $\otimes$-pure.
Having defined a $\class{C}$-acyclic and degreewise $\otimes$-pure subcomplex $P_{\alpha}\subseteq_{\otimes} X$, and assuming that $P_{\alpha}\neq
X$, we let $P_{\alpha+1}/P_{\alpha}\subseteq_{\otimes} X/P_{\alpha}$ be a nonzero $\gamma$-presentable and $\lambda$-pure subcomplex of $X/P_{\alpha}$. Since $P_\alpha$ and $X$ are both $\class{C}$-acyclic, Proposition~\ref{prop-C-acyclic-properties}(2) assures us that $X/P_{\alpha}$ is also in ${}_\class{C}\class{W}$. So then $P_{\alpha+1}/P_{\alpha}$ is $\class{C}$-acyclic as well by Lemma~\ref{lemma-lambda-pure-subs}.  Since both $P_{\alpha+1}/P_\alpha \subseteq_{\otimes} X/P_\alpha$ and $P_\alpha \subseteq_{\otimes} X$, we infer from Lemma~\ref{lemma-purity-prop} that $P_{\alpha+1}\subseteq_{\otimes} X$ too. We want to show that $P_{\alpha+1}$ is $\class{C}$-acyclic too. For this we first note  $(X/P_\alpha)/(P_{\alpha+1}/P_\alpha) \cong X/P_{\alpha+1}$ must be in ${}_\class{C}\class{W}$, by Lemma~\ref{lemma-lambda-pure-subs} (or even by Proposition~\ref{prop-C-acyclic-properties}(2)).  It then follows from Proposition~\ref{prop-C-acyclic-properties}(2) that $P_{\alpha+1}$ is $\class{C}$-acyclic. For the limit ordinal
step, we define $P_{\beta}= \bigcup_{\alpha <\beta}P_{\alpha}$; this
is a colimit of degreewise $\otimes$-pure subcomplexes of $X$, so is also a degreewise $\otimes$-pure subcomplex of
$X$. Of course $P_{\beta}$ must also be $\class{C}$-acyclic by Proposition~\ref{prop-C-acyclic-properties}(1). This process will eventually stop when $P_{\alpha}=X$, at which point
we have written $X$ as a transfinite extension in $\cha{G}_{\otimes}$ of complexes in $\class{S}$. So ${}_\class{C}\class{W} \subseteq \leftperp{(\rightperp{\class{S}})}$.

This proves $({}_\class{C}\class{W}, \class{F}) = (\leftperp{(\rightperp{\class{S}})}, \rightperp{\class{S}})$ is a complete cotorsion pair in $\cha{G}_{\otimes}$. We showed in Proposition~\ref{prop-C-acyclic-properties} that ${}_\class{C}\class{W}$ is thick in $\cha{G}_{\otimes}$.

For (2). Since $D^n(M)$ is $\class{C}$-acyclic for any $\cat{G}$-object $M$, the isomorphism $$\Ext^1_{\otimes}(D^n(M),Y) \cong \Ext^1_{\mathcal{P}_{\otimes}}(M,Y_n)$$ implies that each $Y_n$ is $\otimes$-pure injective whenever $F \in \class{F}$. From this we deduce $F \in \class{F}$ if and only if each $F_n$ is $\otimes$-pure injective and every chain map $X \xrightarrow{} F$ is null homotopic whenever $X \in {}_\class{C}\class{W}$.

For (3), it follows from (2) that $X \in  {}_\class{C}\class{W}\cap\class{F}$  if and only if $X$ has $\otimes$-pure injective components and $X \xrightarrow{1_X} X$ is null homotopic. As in~\cite[Section 3.2]{estrada-gillespie-odabasi}, this class coincides with the class $\tilclass{PI}_{\otimes}$ of injective objects relative to $\cha{G}_{\otimes}$.
\end{proof}

In the language of abelian model structures we have shown the following.

\begin{corollary}\label{corollary-C-acyclic-models}
Let $(\cat{G}, \otimes)$ be any closed symmetric monoidal Grothendieck category.
Let $\class{C}$ be any class of $\cat{G}$-chain complexes and ${}_\class{C}\class{W}$ the class of all $\class{C}$-acyclic complexes.
Then $(All, {}_\class{C}\class{W}, \class{F})$ is a Hovey triple with respect to the exact structure $\cha{G}_{\otimes}$. The corresponding  model structure is cofibrantly generated.
\end{corollary}

The existence of $\class{C}$-acyclic covers is also rather immediate.

\begin{corollary}\label{corollary-C-acyclic-covers}
Let $(\cat{G}, \otimes)$ be any closed symmetric monoidal Grothendieck category. Then any chain complex $X \in \cha{G}$ has a $\class{C}$-acyclic cover.
\end{corollary}

\begin{proof}
By the theorem, every complex has a special $\class{C}$-acyclic  precover and by Proposition~\ref{prop-C-acyclic-properties}, the class ${}_\class{C}\class{W}$ of all $\class{C}$-acyclic complexes is closed under direct limits. It follows from a general result, for example see~\cite[Prop. 3.9]{estrada-guil-odabasi-phantom}, that ${}_\class{C}\class{W}$ is a covering class.
\end{proof}

\section{K-flat covers and the K-flat derived category}\label{sec-k-flat-pure-ac}

In this section we consider the case of K-flat and pure acyclic complexes. Again, $(\cat{G}, \otimes)$ denotes any closed symmetric monoidal Grothendieck category throughout.

\begin{definition}\label{def-k-flat}
A chain complex  $X$ is \emph{K-flat} if $E \tensor X$ is acyclic for all acyclic chain complexes $E$. We denote the class of all K-flat complexes by $\class{KF}$.
\end{definition}

Building on Theorem~\ref{C-acyclics-cotorsion-pair} and Corollary~\ref{corollary-C-acyclic-models} we have our main application: A cofibrantly generated model for the Verdier quotient $K(\cat{G})/\class{KF}$. One might refer to this as the \emph{K-flat derived category} of $(\cat{G},\otimes)$.

\begin{theorem}\label{theorem-K-derived-category}
Let $(\cat{G}, \otimes)$ be any closed symmetric monoidal Grothendieck category.  Then the class $\class{KF}$ of all K-flat complexes is a thick subcategory of $K(\cat{G})$. The Verdier quotient, $K(\cat{G})/\class{KF}$, is a well generated triangulated category and is equivalent to the homotopy category associated to a Hovey triple
$(All, \class{KF}, \class{I})$ on the exact category $\cha{G}_{\otimes}$.
This model structure is cofibrantly generated and satisfies the following:
\begin{enumerate}
\item $\class{I}$ is the class of all complexes with  $\otimes$-pure injective components and such that all chain maps $X \xrightarrow{} I$ are null homotopic whenever $X$ is K-flat.
\item A chain map $f : X \xrightarrow{} Y$ is a weak equivalence in this model structure if and only if its mapping cone is K-flat, equivalently, if and only if $E \otimes X \xrightarrow{E \otimes f} E \otimes Y$ is a homology isomorphism for all acyclic complexes $E$.
\item We have a triangle equivalence, $K(\cat{G})/\class{KF} \cong K(\class{I})$, where $K(\class{I})$ is the strictly full subcategory of $K(\cat{G})$ generated by $\class{I}$.
\end{enumerate}
\end{theorem}

\begin{proof}
It follows from Theorem~\ref{C-acyclics-cotorsion-pair} and Corollary~\ref{corollary-C-acyclic-models}. Arguments similar to those in \cite[Theorem~4.4]{gillespie-K-flat} and  \cite[Theorem~4.1(ii)]{gillespie-ac-pure-proj} will show that the weak equivalences are as described.
\end{proof}

Of course we also have the following special case of Corollary~\ref{corollary-C-acyclic-covers}.

\begin{corollary}\label{corollary-K-flat-covers}
Let $(\cat{G}, \otimes)$ be any closed symmetric monoidal Grothendieck category. Then the class $\class{KF}$ of all K-flat complexes is a covering class.

In particular, for any scheme $\mathbb{X}$, any chain complex of quasi-coherent sheaves has a K-flat cover.
\end{corollary}

A special case of K-flat complexes are the pure acyclic complexes. 
\begin{definition}\label{def-pure-acyclic}
A chain complex $X$ is called \textbf{$\otimes$-pure acyclic} (or simply \textbf{pure acyclic}) if $C \tensor X$ is acyclic for \emph{all} complexes $C \in \cha{G}$.
\end{definition}
We prove in Proposition~\ref{prop-pur-acyc} that it is enough to only require that $M \tensor X := S^0(M) \tensor X$ is acyclic for all objects $M \in \cat{G}$. Our proof will use the following lemma. Here $\class{S}$ denotes a set which cogenerates the $\otimes$-pure injective cotorsion pair $(\cat{G}, \class{PI}_{\otimes})$. Such a set must exist; see~\cite[Lemma~3.6 and Prop.~3.7]{estrada-gillespie-odabasi}.

\begin{lemma}\label{lemma-injective-cot-pair}
 Let $\class{S}$ be a set cogenerating the $\otimes$-pure injective cotorsion pair $(\cat{G}, \class{PI}_{\otimes})$ in the exact category $(\cat{G}, \mathcal{P}_{\otimes})$. Then the injective cotorsion pair $(\cha{G}, \tilclass{PI}_{\otimes})$ in the exact category $\cha{G}_\otimes$ (see \cite[Prop.~3.11]{estrada-gillespie-odabasi}) is cogenerated by the set
 $$\class{X} := \{S^n(M) \,|\, M \in \class{S}, n \in \Z \}.$$
 Consequently, every chain complex $X \in \cha{G}$ is a direct summand
 of a transfinite degreewise $\tensor$-pure extension (i.e. a transfinite extension in the exact category $\cha{G}_\otimes$) of complexes in the set
 $$\class{X} = \{S^n(M) \,|\, M \in \class{S}, n \in \Z \} \cup \{D^n(L_i)\}$$
 where $\{L_i\}$ is a set of generators for the pure exact structure $\class{P}_\otimes$.
\end{lemma}

\begin{proof}
Suppose $Y \in \rightperp{\class{X}}$, the right orthogonal in the exact category $\cha{G}_{\otimes}$. For each $L_i \in \{L_i\}$ we have an obvious short exact sequence
$$0 \xrightarrow{} S^{n-1}(L_i) \xrightarrow{} D^n(L_i) \xrightarrow{} S^n(L_i) \xrightarrow{} 0$$
in the exact category $\cha{G}_\otimes$. Applying $\Hom_{\cha{G}}(-,Y)$ we obtain an exact sequence of abelian groups
$$ \Hom_{\cha{G}}(D^n(L_i) ,Y) \xrightarrow{} \Hom_{\cha{G}}(S^{n-1}(L_i)  ,Y) \xrightarrow{} \Ext^1_{\otimes}(S^n(L_i),Y) = 0.$$
It follows that $\Hom_{\cat{G}}(L_i ,Y_n) \xrightarrow{\Hom_{\cat{G}}(L_i,d_n)} \Hom_{\cat{G}}(L_i,Z_{n-1}Y)$ is an epimorphism. Since $\{L_i\}$ is a set of generators for the pure exact structure $\class{P}_\otimes$, it means each
$0 \xrightarrow{} Z_nY \xrightarrow{} Y_n \xrightarrow{d_n} Z_{n-1}Y \xrightarrow{} 0$ is a $\otimes$-pure exact sequence. Being a $\otimes$-pure exact complex, $Y$ must be exact in the ordinary sense, and so we have a standard isomorphism of Yoneda Ext groups
$$\Ext^1_{\cha{G}}(S^n(M),Y) \cong \Ext^1_{\cat{G}}(M,Z_nY)$$
for any object $M \in \cat{G}$. For example, there is proof of this in~\cite[Lemma~4.2]{gillespie-degreewise-model-strucs}. Following the proof there, one can check that, since $Y$ is $\otimes$-pure exact, the isomorphism restricts to an isomorphism of the subgroups
$$\Ext^1_{\otimes}(S^n(M),Y) \cong \Ext^1_{\class{P}_\otimes}(M,Z_nY).$$
In particular, we have
$$\Ext^1_{\class{P}_\otimes}(M,Z_nY) \cong \Ext^1_{\otimes}(S^n(M),Y) = 0$$ for each $M \in \class{S}$.  This proves $Z_nY$ is $\otimes$-pure injective, and thus $Y \in \tilclass{PI}_{\otimes}$.
\end{proof}

\begin{proposition}\label{prop-pur-acyc}
A complex $X$ is pure acyclic (in the strong sense of Definition \ref{def-pure-acyclic}) if and only if $M \tensor X := S^0(M) \tensor X$ is acyclic for all objects $M \in \cat{G}$.
\end{proposition}

\begin{proof}
For any chain complex $X$, Proposition~\ref{prop-C-acyclic-properties} implies (by taking $\class{C} = \{X\}$) that the class ${}_X\class{W}$, of all chain complexes $W$ for which $X \otimes W \cong W \otimes X$ is acyclic, is closed under direct summands and transfinite extensions in the exact category  $\cha{G}_\otimes$. If we assume $X$ is a complex for which $M \tensor X := S^0(M) \tensor X$ is acyclic for all objects $M \in \cat{G}$, then certainly $S^n(M) \in {}_X\class{W}$ for all $n$ and $M \in \cat{G}$. It follows easily that also $D^n(M) \in {}_X\class{W}$ for all $n$ and $M \in \cat{G}$, see for example~\cite[Exercise~1.2.5]{weibel}. But by Lemma~\ref{lemma-injective-cot-pair}, every complex must be a direct summand of a transfinite extension in the exact category $\cha{G}_\otimes$ of such sphere or disk complexes. So ${}_X\class{W}$ must be the class of all complexes, which means $X$ is pure acyclic in the sense of Definition \ref{def-pure-acyclic}.
\end{proof}

\section{Consequences of $\otimes$-flat generators}\label{sec-flat-generators}

We continue to let $(\cat{G}, \otimes)$ denote a closed symmetric monoidal Grothendieck category, and $\class{KF}$ denote the class of K-flat complexes.
An object $F \in \cat{G}$ will be called $\otimes$-flat if the functor $F\otimes -$ is exact. Note that a sphere complex $S^n(F)$ is K-flat if and only if $F$ is $\otimes$-flat. In this section we look at the consequences of the  assumption that $\cat{G}$ possesses a set of generators, $\{F_i\}$, with each $F_i$ a $\tensor$-flat object. For short, we will express this by saying \emph{$\cat{G}$ has a set of $\otimes$-flat generators}.

\begin{lemma}\label{lemma-K-flat-perp}
Assume $\cat{G}$ has a set of $\otimes$-flat generators. Then every complex in $\rightperp{\class{KF}}$, the right Ext-orthogonal  in the exact category $\cha{G}_{\otimes}$, must be acyclic.
\end{lemma}

\begin{proof}
Let $\{F_i\}$ be a set of $\tensor$-flat generators for $\cat{G}$.
We already know that any $E \in \rightperp{\class{KF}}$ must have $\otimes$-pure injective components. So for each $F_i$, we have
$$0 = \Ext^1_{\otimes}(S^n(F_i),E) = \Ext^1_{dw}(S^n(F_i),E) \cong H_{n-1}[\Hom_{\cat{G}}(F_i,E)].$$
So each complex $\Hom_{\cat{G}}(F_i,E)$ is acyclic, and the assumption that $\{F_i\}$ is a generating set implies $E$ is acyclic.
\end{proof}

\begin{proposition}\label{prop-K-flat-acyclic}
Assume $\cat{G}$ has a set of $\otimes$-flat generators.
$X$ is both acyclic and K-flat if and only if $X$ is pure acyclic.
\end{proposition}

\begin{proof}
$(\impliedby)$ This direction is clear, for if $I$ is the unit of the monoidal structure, then $S^0(I)\otimes X = X$ is acyclic.
For $(\implies)$, assume $X$ is both acyclic and K-flat. Let $C \in \cha{G}$ be arbitrary. By Theorem~\ref{theorem-K-derived-category} we can find a degreewise $\otimes$-pure sequence
$0 \xrightarrow{} E \xrightarrow{} K \xrightarrow{} C \xrightarrow{} 0$ where $K \in\class{KF}$ and $E \in \rightperp{\class{KF}}$. Lemma~\ref{lemma-degreewise-purity} promises we get another short exact sequence of complexes $0 \xrightarrow{} E \otimes X \xrightarrow{} K\otimes X \xrightarrow{} C\otimes X \xrightarrow{} 0$. By the assumption that $X$ is K-flat, and by Lemma~\ref{lemma-K-flat-perp}, we obtain that $E \otimes X$ is acyclic.  By the assumption that $X$ is acyclic, and because $K$ is K-flat, we obtain that $K\otimes X \cong X\otimes K$ is acyclic too. It follows from the 2 out of 3 property that $C\otimes X$ is also acyclic, proving $X$ is pure acyclic.
\end{proof}

To state the next result we set $\class{D}_{\textnormal{K-flat}}(\cat{G}) := K(\cat{G})/\class{KF}$.

\begin{theorem}\label{theorem-recollement}
Let $(\cat{G}, \otimes)$ be any closed symmetric monoidal Grothendieck category possessing a set of $\otimes$-flat generators.
We have a recollement of well generated triangulated categories
\[
\xy
(-28,0)*+{\class{D}_{\textnormal{K-flat}}(\cat{G})};
(0,0)*+{\class{D}_{\otimes\textnormal{-pur}}(\cat{G})};
(25,0)*+{\class{D}(\cat{G})};
{(-19,0) \ar (-10,0)};
{(-10,0) \ar@<0.5em> (-19,0)};
{(-10,0) \ar@<-0.5em> (-19,0)};
{(10,0) \ar (19,0)};
{(19,0) \ar@<0.5em> (10,0)};
{(19,0) \ar@<-0.5em> (10,0)};
\endxy
.\]
\end{theorem}

\begin{proof}
From Theorem~\ref{theorem-K-derived-category}, we have the K-flat model, $\mathfrak{M}_2 = (All, \class{KF}, \class{I})$, whose homotopy category is $\class{D}_{\textnormal{K-flat}}(\cat{G})$.  Taking $\class{C}$ to be the class of all chain complexes,  Corollary~\ref{corollary-C-acyclic-models} provides an abelian model structure $\mathfrak{M}_1 = (All, \class{E}_{\otimes}, \class{F}_1)$ on the exact category $\cha{G}_{\otimes}$. By Proposition~\ref{prop-pur-acyc} its homotopy category is precisely $\class{D}_{\otimes\textnormal{-pur}}(\cat{G})$, the $\otimes$-pure derived category of~\cite[Theorem~A]{estrada-gillespie-odabasi}.
Taking $\class{C} = \{S^0(I)\}$, where $I$ is the unit of the monoidal structure, Corollary~\ref{corollary-C-acyclic-models} provides an abelian model structure $\mathfrak{M}_3 = (All, \class{E}, \class{F}_3)$, where $\class{E}$ is the class of all (usual) acyclic chain complexes.  From the description of the fibrant objects given in Theorem~\ref{C-acyclics-cotorsion-pair}, the class $\class{F}_3$ must be the class of all K-injective complexes with $\otimes$-pure injective components. Its homotopy category is $\class{D}(\cat{G})$, the usual derived category of $\cat{G}$. These three abelian model structures are cofibrantly generated and so their homotopy categories are each well generated triangulated categories. The recollement follows at once from Proposition~\ref{prop-K-flat-acyclic} and~\cite[Theorem~4.6]{gillespie-recollement}
\end{proof}

Next, we note that a monoidal model structure for the usual derived category, based on the K-flats, always exists when we have  a set of $\otimes$-flat generators.
In the following statement, we keep the same notation as in the proof of Theorem~\ref{theorem-recollement}.

\begin{theorem}\label{theorem-K-flat-model-derived-cat}
Let $(\cat{G}, \otimes)$ be any closed symmetric monoidal Grothendieck category possessing a set of $\otimes$-flat generators.
Then
$$\mathfrak{M}^{Kflat}_{\textnormal{der}} := (\class{KF}, \class{E}, \class{F}_1)$$
is a Hovey triple relative to the exact category  $\cha{G}_{\otimes}$. It corresponds to a cofibrantly generated model structure for the usual derived category, $\class{D}(\cat{G})$.
The model structure is monoidal and satisfies the monoid axiom as long as the unit $I$ is $\otimes$-flat. Moreover, the model structure $$\mathfrak{M}_{\otimes\textnormal{-pur}} := (All, \class{E}_{\otimes}, \class{F}_1)$$
for $\class{D}_{\otimes\textnormal{-pur}}(\cat{G})$, the $\otimes$-pure derived category, is also monoidal  and satisfies the monoid axiom  (even if the unit $I$ is not $\otimes$-flat).
\end{theorem}

\begin{proof}
We have $\class{KF}\cap \class{E} = \class{E}_{\otimes}$, the class of all ($\otimes$-)pure acyclic complexes, by Proposition~\ref{prop-K-flat-acyclic}. It follows from this that we also have $\class{E} \cap \class{F}_1 = \class{I}$, the right orthogonal to $\class{KF}$.   Indeed $\class{I}\subseteq \class{E} \cap \class{F}_1$ by Lemma~\ref{lemma-K-flat-perp}. To show $\class{E} \cap \class{F}_1\subseteq \class{I}$, embed a complex $X\in \class{E} \cap \class{F}_1$ as a $\otimes$-pure subobject of a complex in $\class{I}$ with K-flat cokernel, and note that the embedding must be split exact.

It is easy to check that the four monoidal conditions of~\cite[Theorem~7.2]{hovey} hold for both $\mathfrak{M}^{Kflat}_{\textnormal{der}}$ and $\mathfrak{M}_{\otimes\textnormal{-pur}}$.   In particular, all cofibrations are $\otimes$-pure by Lemma~\ref{lemma-degreewise-purity}. And $S^0(I)$, the unit for the monoidal structure on $\cha{G}$, is K-flat whenever $I$ is $\otimes$-flat. By associativity of $\otimes$, the class $\class{KF}$ is closed under $\otimes$ and $\class{E}_{\otimes}$ is a $\otimes$-ideal. The monoid axiom is verified by the two conditions given in~\cite[Theorem~7.4]{hovey}. In particular, pure acyclic complexes are closed under transfinite compositions of admissible monics in $\cha{G}_{\otimes}$.
\end{proof}

\section{The right orthogonal to K-flats}\label{sec-right-orthogonal}

For modules over a ring $R$, Emmanouil shows in~\cite{emmanouil-K-flatness-and-orthogonality-in-homotopy-cats} that $\rightperp{\class{KF}}$, in $K(R)$, is  the class of acyclic complexes of pure injective $R$-modules. Using the language of cotorsion pairs and Ext-orthogonals in $\ch_{\otimes}$ this is described in~\cite[Them~4.4]{gillespie-K-flat}. In this section we would like to find conditions that allow us to generalize this result. In particular, we would like this to be true for complexes of quasi-coherent sheaves. We prove here that this is indeed the case whenever the underlying scheme is quasi-compact and semiseparated. We first prove some general results that appear to be useful in other categories, and for other cotorsion pairs.

\subsection{General Results}

Here we prove some general lemmas, and Theorem~\ref{theorem-direct-limits}, which can be useful to characterize the right Ext-orthogonal of a class of chain complexes.

\begin{lemma}\label{lemma-direct-limits-cot-pair}
Let $\cat{G}$ be a Grothendieck category, considered along with a proper class $\class{E}$ of short exact sequences that is closed under direct limits.
Assume $\class{F}$ is a class of objects that is closed under direct summands, $\class{E}$-extensions, and direct limits. Let $\class{E}_{\class{F}}$ denote the inherited Quillen exact structure on $\class{F}$. If $(\class{X}, \class{Y})$ is a cotorsion pair on the exact category $(\class{F},\class{E}_{\class{F}})$, and $\class{X}$ is thick relative to $\class{E}_{\class{F}}$, then  $\class{X}$ is closed under direct limits.
\end{lemma}

Our proof is merely a verification that the one in~\cite[Prop.~3.1/3.2]{gillespie-ding-modules} generalizes to such a cotorsion pair $(\class{X},\class{Y})$. Alternatively, the result follows from a variation of this type of argument which is clearly written down in Positselski and \v{S}t'ov\'{\i}\v{c}ek \cite[Prop.~5.1]{posit-stov-flat-sheaves}.

\begin{proof}
The argument from~\cite[Prop.~3.1/3.2]{gillespie-ding-modules} generalizes to this setting.  Indeed the proof there analyzes a short exact sequence
\begin{equation}\label{eq-ses}
0 \xrightarrow{} K \xrightarrow{\subseteq} \bigoplus_{i<\lambda} X_i \xrightarrow{}  \varinjlim_{i<\lambda} X_i \xrightarrow{}  0
\end{equation}
and $K$ is shown to be a directed union of subobjects $K_J$, each of which is (isomorphic to) a direct summand of $\bigoplus_{i<\lambda} X_i \in \class{X}$. In particular, each inclusion $K_J \subseteq \bigoplus_{i<\lambda} X_i$  is an admissible $\class{E}_{\class{F}}$-monic since it is a split monic.  We would like to verify that the reduction in the proof that expresses $K = \bigcup_{J\in\class{S}} K_J$ as a direct union of a well-ordered continuous  (smooth) chain  (via~\cite[Lemma1.6/ Corollary1.7/ Remark]{AR}), allows us to still conclude $K \in \class{X}$.  This will follow if we are able to verify the
\begin{equation}
\noindent\emph{Claim: $\class{X}$ is closed under direct unions of $\class{E}_{\class{F}}$-admissible subobjects.}
\end{equation}
So let $\{X_i\subseteq X\}_{i\in I}$ be a direct system of subobjects, indexed by some directed set $(I,\leq)$, with each $X_i\in \class{X}$. We proceed by transfinite induction on the cardinality $\gamma = |I|$ of the directed indexing set $(I,\leq)$, to show that the direct union $K := \bigcup_{i\in I} X_i$ must be in the class $\class{X}$.

First, we note that if $\gamma$ is finite, then the direct union $K := \bigcup_{i\in I} X_i$ must coincide with the particular $X_i$ representing the unique maximal element of $I$. So clearly $K = X_i\in\class{X}$ in this case.

For the induction step, let us be given such a direct union $K = \bigcup_{i\in I} X_i$ with $\gamma = |I|$ an infinite cardinal, and suppose that all such direct unions are in $\class{X}$ whenever they are indexed by a directed set of cardinality less than $\gamma$. Following the proof of~\cite[Corollary1.7]{AR}, we may re-express $K$ as a well-ordered continuous direct union, $K = \bigcup_{k<\gamma} X_{I_k}$,  where $\gamma = |I|$ and, by construction, each $X_{I_k} = \bigcup_{i\in I_k} X_{i}$ is itself a direct union of some of the subojects $X_i$, over a directed subset $I_k \subseteq I$ of \emph{smaller} cardinality $|I_k|<\gamma$. By the induction hypothesis, each $X_{I_k} \in \class{X}$. Moreover, the inclusion $X_{I_k}  \subseteq X$ must also be an admissible monic for $\class{E}_{\class{F}}$, since each inclusion $X_i \subseteq X$ is such and by our other assumptions which imply $\class{E}_{\class{F}}$ is closed under direct limits.  It then  follows from~\cite[Prop.~7.6 (dual)]{buhler-exact categories} that each inclusion $X_{I_k} \subseteq X_{I_{k +1}}$ is an admissible monic in $\class{E}_{\class{F}}$.  Since $\class{X}$ is thick relative to $\class{E}_{\class{F}}$, we get that $X_{I_{k +1}}/X_{I_k} \in \class{X}$ for all $k<\gamma$. In other words, $K = \bigcup_{k<\gamma} X_{I_k}$ is a transfinite $\gamma$-extension of all the objects $X_{I_{k +1}}/X_{I_k} \in \class{X}$. So it follows from Eklof's Lemma (the version in~\cite[Prop.~2.12]{saorin-stovicek} applies), that $K\in\class{X}$.

This completes the proof of the above \emph{Claim}. Turning back to the short exact sequence in (\ref{eq-ses}), we have shown $K\in\class{X}$. Moreover, this sequence must be in $\class{E}_{\class{F}}$ by our  assumptions on direct limits. So, since $\class{X}$ is thick in $\class{E}_{\class{F}}$, it follows  that  $\varinjlim_{i<\lambda} X_i \in \class{X}$.
\end{proof}

Given an exact category $(\class{A},\class{E})$, let $\cha{A}_{\class{E}}$ denote the exact category of chain complexes along with the inherited degreewise exact structure.  Let us say that a cotorsion pair in $\cha{A}_{\class{E}}$ is \emph{suspension closed} if each class (equivalently, either class) is closed under $\Sigma$ and $\Sigma^{-1}$. Given a class of objects $\class{X}\subseteq \cat{A}$, we will use the notation $\dwclass{X}$ to denote the class of all chain complexes that are degreewise in $\class{X}$.

\begin{lemma}\label{lemma-restricted-cot-pair}
Let $(\class{X},\class{Y})$ be a complete hereditary cotorsion pair on an exact category $(\cat{A},\class{E})$. Assume  $(\hat{\class{X}},\hat{\class{Y}})$ is a suspension closed complete cotorsion pair in $\cha{A}_{\class{E}}$ such that  every complex of $\hat{\class{X}}$ has components in $\class{X}$ and every complex of $\hat{\class{Y}}$ has components in $\class{Y}$. Then $(\hat{\class{X}},\hat{\class{Y}})$ restricts to a complete cotorsion pair $(\hat{\class{X}},\hat{\class{Y}}\cap\dwclass{X})$ in the exact category $\cha{X}_{\class{E}_{\class{X}}}$, and moreover $\hat{\class{X}}$ is thick in $\cha{X}_{\class{E}_{\class{X}}}$.
\end{lemma}

\begin{proof}
The key step is to show that $\leftperp{[\hat{\class{Y}}\cap\dwclass{X}]}$, taken in $\cha{X}_{\class{E}_{\class{X}}}$, is contained within $\hat{\class{X}}$. So let $X \in \cha{X}$ be in $\leftperp{[\hat{\class{Y}}\cap\dwclass{X}]}$.  Using that $(\hat{\class{X}},\hat{\class{Y}})$ is a complete cotorsion pair, we may write a short exact sequence $0 \xrightarrow{} Y \xrightarrow{}  X' \xrightarrow{}  X \xrightarrow{}  0$ where $Y \in \hat{\class{Y}}$ and $X' \in \hat{\class{X}}$. Since all $X_n,X'_n \in \class{X}$ and  $(\class{X},\class{Y})$ is  hereditary, we get $Y_n \in \class{X}$. So $Y\in  \hat{\class{Y}}\cap\dwclass{X}$ and the sequence splits and we conclude $X \in \hat{\class{X}}$. Since $(\hat{\class{X}},\hat{\class{Y}})$ is suspension closed and degreewise $\Ext^1_{\class{E}}$-orthogonal, we see that $X \in \hat{\class{X}}$ if and only if $\homcomplex(X,Y)$ is acyclic for all $Y \in \hat{\class{Y}}$. Moreover, given a short exact sequence $0 \xrightarrow{} X \xrightarrow{}  X' \xrightarrow{}  X'' \xrightarrow{}  0$ in $\cha{X}_{\class{E}_{\class{X}}}$, then
$$0 \xrightarrow{} \homcomplex(X,Y) \xrightarrow{}  \homcomplex(X',Y) \xrightarrow{}  \homcomplex(X'',Y) \xrightarrow{}  0$$ is a short exact sequence of chain complexes of abelian groups for all  $Y \in \hat{\class{Y}}$. So  $\hat{\class{X}}$ has the 2 out of 3 property relative to the exact category  $\cha{X}_{\class{E}_{\class{X}}}$.
\end{proof}

The next theorem will also use the following notation: (i) for a given class of objects $\class{D}$ in an additive  category with direct limits, we let $\varinjlim \class{D}$ denote the class of all objects that equal a direct limit of some directed system of objects in $\class{D}$, and (ii) assuming we are in an exact category with $\class{E}$ denoting the class of all short exact sequences, we let $\langle \class{D}\rangle_{\class{E}}$, or just $\langle \class{D}\rangle$ if the context is clear, denote the class of all objects $X$ possessing either a finite resolution, or a finite coresolutions, by objects in $\class{D}$.

\begin{theorem}\label{theorem-direct-limits}
Let $\cat{G}$ be a Grothendieck category, considered along with a proper class $\class{E}$ of short exact sequences that is closed under direct limits.
Suppose $(\class{X},\class{Y})$ is a complete hereditary cotorsion pair on $(\cat{G},\class{E})$, cogenerated by a (generating) set, and that $\class{X}$ is closed under direct limits in $\cat{G}$.  Then:
\begin{enumerate}
\item  $(\leftperp{\dwclass{Y}},\dwclass{Y})$ is a complete hereditary cotorsion pair in the exact category $\cha{G}_{\class{E}}$. The class  $\leftperp{\dwclass{Y}}$  is closed under direct limits, and it is thick relative to the exact subcategory $\cha{X}_{\class{E_{\class{X}}}}$.
\item Let $\class{T}$ be the class of all contractible complexes with components in $\class{X}$. Then for any complex $X \in \langle\varinjlim \class{T}\rangle_{\cha{X}_{\class{E_{\class{X}}}}}$, we have  $$\Ext^1_{\cha{G}_{\class{E}}}(X,Y) = \Ext^1_{dw}(X,Y) = 0$$ for all  $Y \in \dwclass{Y}$. Equivalently, any chain map $f : X \xrightarrow{} Y$  is null homotopic whenever $X \in \langle\varinjlim \class{T}\rangle_{\cha{X}_{\class{E_{\class{X}}}}}$ and $Y \in \dwclass{Y}$.
\end{enumerate}
\end{theorem}

\begin{proof}
Taking the set of all disks $\{D^n(S)\}$ as $S$ ranges through a  cogenerating set for $(\class{X},\class{Y})$, then it cogenerates a complete hereditary cotorsion pair, $(\leftperp{\dwclass{Y}},\dwclass{Y})$, in $\cha{G}_{\class{E}}$. It satisfies the hypotheses of Lemma~\ref{lemma-restricted-cot-pair}, and so it restricts to a (complete) cotorsion pair,  $$(\leftperp{\dwclass{Y}}, \dwclass{Y}\cap\dwclass{X}),$$ on the exact subcategory $\cha{X}_{\class{E_{\class{X}}}}$, and $\leftperp{\dwclass{Y}}$ must be thick in $\cha{X}_{\class{E_{\class{X}}}}$. Moreover, by our assumptions, $\cha{X}_{\class{E_{\class{X}}}}$ is closed under direct limits. So referring to Lemma~\ref{lemma-direct-limits-cot-pair} we may take $(\cat{G},\class{E})$ to be $\cha{G}_{\class{E}}$, take $(\class{F},\class{E}_{\class{F}})$ to be $\cha{X}_{\class{E_{\class{X}}}}$, and take the cotorsion pair $(\class{X}, \class{Y})$ to be $(\leftperp{\dwclass{Y}}, \dwclass{Y}\cap\dwclass{X})$. We conclude that $\leftperp{\dwclass{Y}}$ is closed under direct limits, proving (1).

Now let $\class{T}$ be the class of all contractible complexes with components in $\class{X}$. Then certainly $\class{T} \subseteq \leftperp{\dwclass{Y}}$, so $\varinjlim \class{T} \subseteq \leftperp{\dwclass{Y}}$ by what we just proved.   And since $\leftperp{\dwclass{Y}}$ is thick in $\cha{X}_{\class{E_{\class{X}}}}$, we get that $\langle\varinjlim \class{T}\rangle_{\cha{X}_{\class{E_{\class{X}}}}} \subseteq \leftperp{\dwclass{Y}}$. It means   $\Ext^1_{\cha{G}_{\class{E}}}(X,Y) = 0$   for any complex $X \in \langle\varinjlim \class{T}\rangle_{\cha{X}_{\class{E_{\class{X}}}}}$,  and $Y \in \dwclass{Y}$.
\end{proof}

To show the power of the above theorem we show how it may be used to recover an important result of  Bazzoni, Cort\'es-Izurdiaga, and  Estrada from~\cite[Theorem~5.3]{bazzoni-cortes-estrada}.

\begin{example}[Theorem~5.3 of~\cite{bazzoni-cortes-estrada}]\label{example-flat-cot-pair}
Let $R$ be a ring. Take $(\cat{G},\class{E})$ to be the category of $R$-modules along with the usual abelian exact structure. Take $(\class{X},\class{Y}) = (\class{F},\class{C})$ to be Enochs' flat cotorsion pair. It is known that an acyclic complex with flat cycles is a direct limit of contractible complexes with projective components. Thus $\tilclass{F} = \varinjlim \class{T}$. It follows easily that $\tilclass{F} = \langle\varinjlim \class{T}\rangle_{\cha{F}_{\class{E}_{\class{F}}}}$. Therefore, $(\tilclass{F},\dgclass{C}) = (\tilclass{F},\dwclass{C})$.
\end{example}

We can also use Theorem~\ref{theorem-direct-limits} to recover the result due to Christensen, Estrada, and Thompson from~\cite[Theorem~3.3]{cet-G-flat-stable-scheme}, which is the extention of~\cite[Theorem~5.3]{bazzoni-cortes-estrada} to quasi-coherent sheaves. See Remark~\ref{remark-CET}.
As another example, we have the following motivating result.

\begin{corollary}\label{corollary-direct-limits}
Let $(\cat{G}, \otimes)$ be any closed symmetric monoidal Grothendieck category and $\cha{G}_{\otimes}$ the associated chain complex category along with the degreewise $\otimes$-pure exact structure. Let $\class{PI}_{\otimes}$ denote the class of all $\otimes$-pure injective objects, $\class{E}_{\otimes}$ the class of all ($\otimes$-)pure acyclic complexes, and let $\class{T}$ denote the class of all contractible chain complexes. If  $\class{E}_{\otimes} \subseteq \langle\varinjlim\class{T}\rangle_{\cha{G}_\otimes}$, then $(\class{E}_{\otimes},\dwclass{PI}_{\otimes})$ 
is a complete cotorsion pair in $\cha{G}_{\otimes}$.
\end{corollary}

\begin{proof}
Take $(\cat{G},\class{E}) = (\cat{G},\class{P}_{\otimes})$ to be the proper class of all $\otimes$-pure short exact sequences. Take $(\class{X},\class{Y}) = (\cat{G}, \class{PI}_{\otimes})$ to be the $\otimes$-pure injective cotorsion pair. Then $\cha{X}_{\class{E_{\class{X}}}}$ is nothing more than the exact category $\cha{G}_\otimes$.  We get that $\leftperp{\dwclass{PI}_{\otimes}}$ is closed under direct limits and is thick in $\cha{G}_\otimes$. It also clearly contans all contractible complexes. Therefore, $ \langle\varinjlim\class{T}\rangle_{\cha{G}_\otimes}\subseteq \leftperp{\dwclass{PI}_{\otimes}}$. So if $\class{E}_{\otimes} \subseteq \langle\varinjlim\class{T}\rangle_{\cha{G}_\otimes}$, then it is easy to see that the complete cotorsion pair $(\class{E}_{\otimes},\class{F}_1)$ in $\cha{G}_{\otimes}$ coincides with $(\leftperp{\dwclass{PI}_{\otimes}},\dwclass{PI}_{\otimes})$.
\end{proof}

\subsection{Acyclic complexes of $\otimes$-pure quasi-coherent sheaves}

Throughout this section, we assume $\mathbb{X} = (X,\mathcal{O}_X)$ is a quasi-compact and semiseparated scheme. We let $\mathbb{X}$-Mod denote the category of quasi-coherent sheaves on $\mathbb{X}$. It is a closed symmetric monoidal category under the usual sheaf tensor product and internal hom coming from applying the \emph{coherator} to the usual sheafhom. We let $\textnormal{Ch}(\mathbb{X})$ denote the category of chain complexes of quasi-coherent sheaves on $\mathbb{X}$. It inherits the standard symmetric monoidal structure as indicated previously, before Lemma~\ref{lemma-degreewise-purity}.

\begin{theorem}\label{theorem-qc-sheaves}
Let $\class{PI}_{\otimes}$ denote the class of all $\otimes$-pure injective quasi-coherent sheaves on $\mathbb{X}$. Let $\class{E}_{\otimes}$ denote the class of all ($\otimes$-)pure acyclic complexes in $\textnormal{Ch}(\mathbb{X})$, and $\class{KF}$ denote the class of all K-flat complexes of quasi-coherent sheaves. Then:
\begin{enumerate}
\item $(\class{E}_{\otimes},\dwclass{PI}_{\otimes})$ is a complete cotorsion pair in the exact category $\textnormal{Ch}(\mathbb{X})_{\otimes}$.
\item  $(\class{KF},\exclass{PI}_{\otimes})$ is a complete cotorsion pair in the exact category $\textnormal{Ch}(\mathbb{X})_{\otimes}$. Here, $\exclass{PI}_{\otimes}$ denotes the class of all  acyclic (just exact, not necessarily $\otimes$-pure exact) complexes of $\otimes$-pure injectives.
\end{enumerate}
Consequently, $\class{D}_{\otimes\text{-pur}}(\mathbb{X}) \cong K(\class{PI}_{\otimes})$, and $\class{D}_{\textnormal{K-flat}}(\mathbb{X}) \cong K_{ac}(\class{PI}_{\otimes})$, and $\class{D}(\mathbb{X}) \cong K(\class{KF}\cap \dwclass{PI}_{\otimes})$.
\end{theorem}

\begin{proof}
We prove (1) by imitating the argument from~\cite[Theorem~3.3]{cet-G-flat-stable-scheme}, making the necessary adjustments. Let $Y$ be a complex of $\otimes$-pure injective quasi-coherent sheaves. It is enough to show $\Ext^1_{\otimes}(F,Y) = 0$ for all $\otimes$-pure acyclic complexes of quasi-coherent sheaves $F$, as, by~\cite[Corollary 3.10]{estrada-gillespie-odabasi} (or Theorem \ref{C-acyclics-cotorsion-pair}), the containment $\class{E}_{\otimes}^\perp\subseteq \dwclass{PI}_{\otimes}$ always holds. Let $\mathcal{U} = \{U_0, U_1, \cdots , U_d\}$ be a semiseparating open affine covering of $X$ and consider the usual {\v{C}}ech resolution (\cite[Section III.4]{hartshorne}), extended in a degreewise fashion to chain complexes,
\begin{equation}\label{equation-cech-res}
0 \xrightarrow{} F \xrightarrow{\epsilon} \class{C}^0(\mathcal{U}, F) \xrightarrow{} \class{C}^1(\mathcal{U}, F) \xrightarrow{} \cdots \xrightarrow{} \class{C}^d(\mathcal{U}, F) \xrightarrow{} 0.
\end{equation}
Here, $\class{C}^p(\mathcal{U}, F) := \bigoplus_{j_0<j_1<\cdots <j_p}i_*(\widetilde{F(U_{{j_0},{j_1},\cdots,{j_p}})})$, where $j_0<j_1<\cdots <j_p$ ranges over sequences of length $p+1$ in $\{0,1,\cdots , d\}$, and  $U_{{j_0},{j_1},\cdots,{j_p}} := U_{j_0}\cap U_{j_1}\cap\cdots U_{j_p}$, and $i : U_{{j_0},{j_1},\cdots,{j_p}} \xrightarrow{} X$ denotes the inclusion of the open affine into the underlying space of $\mathbb{X}$.  (We are picturing each {\v{C}}ech sheaf complex $\class{C}^p(\mathcal{U}, F)$  in (\ref{equation-cech-res}) as a vertical chain complex.) We note that each such $\class{C}^p(\mathcal{U}, F)$ is exact because $F$ is exact; see~\cite[Section 3.1]{murfet-thesis}. In fact, by~\cite[Prop.~2.10 and the  proceeding remark]{estrada-gillespie-odabasi}, for any tuple of indices, $j_0<j_1<\cdots <j_p$, the complex  $F(U_{{j_0},{j_1},\cdots,{j_p}})$ is a pure acyclic complex of $\mathcal{O}_X(U_{{j_0},{j_1},\cdots,{j_p}})$-modules. As such, each  is isomorphic to a direct limit $$F(U_{{j_0},{j_1},\cdots,{j_p}}) \cong \varinjlim_{\lambda \in \Lambda} W_{\lambda}^{U_{{j_0},{j_1},\cdots,{j_p}}},$$ of contractible chain complexes, $W_{\lambda}^{U_{{j_0},{j_1},\cdots,{j_p}}}$, of $\mathcal{O}_X(U_{{j_0},{j_1},\cdots,{j_p}})$-modules,  by~\cite[Prop.~2.2]{emmanouil-pure-acyclic-complexes}. The direct image functor $i_*$ preserves direct limits and contractible complexes, so it follows that each $i_*(\widetilde{F(U_{{j_0},{j_1},\cdots,{j_p}})})$ is a direct limit of contractible complexes of quasi-coherent sheaves. Letting $\class{T}$ denote the class of all contractible complexes of quasi-coherent sheaves, it means that each $i_*(\widetilde{F(U_{{j_0},{j_1},\cdots,{j_p}})}) \in \varinjlim\class{T}$. It follows that each {\v{C}}ech sheaf complex $\class{C}^p(\mathcal{U}, F)\in \varinjlim\class{T}$ too.
As for the horizontal (degreewise) exactness of  the sequence in (\ref{equation-cech-res}), by construction it is, upon restriction to each open affine $U_i$, degreewise split exact. So it follows from~\cite[Prop.~2.10]{estrada-gillespie-odabasi} that the sequence of complexes in (\ref{equation-cech-res}) is exact relative to the exact structure $\textnormal{Ch}(\mathbb{X})_{\otimes}$. This means that $F \in \langle\varinjlim\class{T}\rangle_{\textnormal{Ch}(\mathbb{X})_{\otimes}}$, and turning to  Corollary~\ref{corollary-direct-limits}, we conclude that $(\class{E}_{\otimes},\dwclass{PI}_{\otimes})$ is a complete cotorsion pair in $\textnormal{Ch}(\mathbb{X})_{\otimes}$.

The second statement follows readily: See the first paragraph of the proof of Theorem~\ref{theorem-K-flat-model-derived-cat}. In fact, we have now shown $\mathfrak{M}^{Kflat}_{\textnormal{der}} := (\class{KF}, \class{E}, \class{F}_1) = (\class{KF}, \class{E}, \dwclass{PI}_{\otimes})$ is a Hovey triple. So $\exclass{PI}_{\otimes} = \class{E}\cap\dwclass{PI}_{\otimes} = \rightperp{\class{KF}}$, in $\textnormal{Ch}(\mathbb{X})_{\otimes}$.
\end{proof}

\begin{remark}\label{remark-CET}
Continuing as in Example~\ref{example-flat-cot-pair}, the extension of~\cite[Theorem~5.3]{bazzoni-cortes-estrada} to complexes of quasi-coherent sheaves was made in~\cite[Theorem~3.3]{cet-G-flat-stable-scheme} (see also \cite[Theorem 7.2]{posit-stov-flat-sheaves}).
We note that the argument in the above proof is just a slight modification of the {\v{C}}ech resolution from their proof. In fact, if the complex $F$ in our proof above has flat components, then the argument is identical to theirs and shows $F \in \langle\varinjlim\class{T}\rangle_{\cha{F}}$, where $\class{T}$ is the class of all contractible complexes of flat quasi-coherent sheaves.
That is, we recover~\cite[Theorem~3.3]{cet-G-flat-stable-scheme} from Theorem~\ref{theorem-direct-limits} applied to the flat cotorsion pair $(\class{F},\class{C})$ in $\mathbb{X}$-Mod in a manner analogous to Example~\ref{example-flat-cot-pair} and the above proof of Theorem~\ref{theorem-qc-sheaves}.
\end{remark}

\subsection{The Affine Case: Equivalence of the Verdier quotients by K-flats and K-absolutely pures}
Here we let $R$ be a ring and $\ch$ denote the category of chain complexes of $R$-modules. $R$ need not be commutative so we consider, say, left $R$-modules. Note that our model structure, $\mathfrak{M}_{\class{KF}} = (All, \class{KF},\exclass{PI})$ on $\ch_{\otimes}$, is \emph{injective} in the sense that it is abelian and every object is cofibrant. It is interesting to note that there is another model structure on $\ch_{\otimes}$, having complexes of pure-projectives as the cofibrant objects, but whose homotopy category is also the K-flat derived category, $\class{D}_{\textnormal{K-flat}}(R) := K(R)/\class{KF}$. Moreover, this model shows that $\class{D}_{\textnormal{K-flat}}(R) $ is equivalent to the Verdier quotient of $K(R)$ by Emmanouil and Kaperonis' \emph{K-absolutely pure} complexes, introduced in~\cite{emmanouil-kaperonis-K-flatness-pure}.

Throughout this section we use the following notation for classes of chain complexes in $\ch$:
\begin{itemize}
\item $\class{E}$ denotes the class of all exact (acyclic) complexes.
\item  $\class{E}_{\otimes}$ denotes the class of all pure acyclic complexes.
\item $\dwclass{PI}$ denotes the class of all complexes that are degreewise pure injective, that is, each component is a pure injective  $R$-module.
\item $\dwclass{PP}$ denotes the class of all complexes that are degreewise pure projective, that is, each component is a pure projective $R$-module.
\item $\class{KF}$ denotes the class of all K-flat complexes.
\item $\class{KA}$ denotes the class of all K-absolutely pure complexes in the sense of~\cite{emmanouil-kaperonis-K-flatness-pure}. (This class was denoted by $\class{V}$ in~\cite{gillespie-ac-pure-proj}.)
\item $\class{F}$ denotes the class of all K-injective complexes with pure injective components; complexes in $\class{F}$ were called \emph{DG-pure injective} in~\cite{gillespie-K-flat}.
\item $\class{C}$ denotes the class of all K-projective complexes with pure projective components;  complexes in $\class{C}$ were called \emph{DG-pure projective} in~\cite{gillespie-ac-pure-proj}.
\end{itemize}

The key to the following result is that, by~\cite[Corollary~3.4]{emmanouil-relation-K-flatness-K-projectivity}, K-flat complexes with pure projective components must also be K-projective. Similarly, K-absolutely pure complexes with pure injective components must be K-injective by~\cite[Prop.~2.3]{emmanouil-kaperonis-K-flatness-pure}.
The proof of the next result will include simple cotorsion theoretic proofs of these two facts.

\begin{proposition}\label{prop-K-flat-K-abspure-rings}
Let $R$ be a ring, and $\ch_{\otimes} \ (= \ch_{pur})$ be the exact category of chain complexes along with the degreewise pure exact structure.
\begin{enumerate}
\item $(\dwclass{PP},\class{KF},\class{E})$ is a cofibrantly generated abelian model structure on $\ch_{\otimes}$. This provides a triangulated equivalence of the Verdier quotient $$K(R)/\class{KF} \cong K_{ac}(\class{PP}),$$ onto the full subcategory $K_{ac}(\class{PP}) \subseteq K(R)$ generated by the class of all acyclic complexes of pure projectives.

\item $(\class{E},\class{KA},\dwclass{PI})$ is a cofibrantly generated abelian model structure on $\ch_{\otimes}$. This provides a triangulated equivalence of the Verdier quotient $$K(R)/\class{KA} \cong K_{ac}(\class{PI}),$$ onto the full subcategory $K_{ac}(\class{PI}) \subseteq K(R)$ generated by the class of all acyclic complexes of pure injectives.
\end{enumerate}
Consequently, the Verdier quotient $K(R)/\class{KF}$ is equivalent to the Verdier quotient $K(R)/\class{KA}$.
\end{proposition}

\begin{proof}
For (1), we already have that $(\dwclass{PP},\class{KF}\cap\class{E}) = (\dwclass{PP},\class{E}_{\otimes})$ is a complete cotorsion pair in $\ch_{\otimes}$, and cogenerated by a set; see~\cite[Cor.~4.2]{gillespie-ac-pure-proj}. Also, $(\class{C},\class{E})$ is a complete cotorsion pair in $\ch_{\otimes}$, and cogenerated by a set, by~\cite[Prop.~6.3]{gillespie-ac-pure-proj}. So we only need to show Emmanouil's relation  $\dwclass{PP}\cap\class{KF} = \class{C}$.

\noindent $(\supseteq)$ This is easy to see.

\noindent $(\subseteq)$ Let $X \in \dwclass{PP}\cap\class{KF}$. Using completeness of $(\class{C},\class{E})$, write a degreewise pure short exact sequence $0 \xrightarrow{} E \xrightarrow{} P \xrightarrow{}X \xrightarrow{} 0$ with $E \in \class{E}$ and $P \in \class{C}$. Since $\class{C} \subseteq \class{KF}$, we have $P \in \class{KF}$. Since $\class{KF}$ is thick in $\ch_{\otimes}$, we have $E \in \class{KF}\cap\class{E} = \class{E}_{\otimes}$. The short exact sequence represents an element of $\Ext^1_{\otimes}(X,E)$. But this Ext group must vanish since $(\dwclass{PP},\class{E}_{\otimes})$ is a cotorsion pair in $\ch_{\otimes}$. Hence the short exact sequence must split, making $X$ is a direct summand of $P$, thus forcing $X$ to also be in $\class{C}$. This completes the proof of (1).

The proof of (2) is dual: The dual results concerning complete cotorsion pairs were already shown in~\cite{gillespie-K-flat}, and $\class{F} = \class{KA}\cap\dwclass{PI}$ follows by a straightforward dual of the above argument.

Finally, we observe that $$K_{ac}(\class{PI}) \cong K(R)/\class{KF} \cong K_{ac}(\class{PP}) \cong K(R)/\class{KA} \cong K_{ac}(\class{PI})$$ The first equivalence is from~\cite{emmanouil-K-flatness-and-orthogonality-in-homotopy-cats}, the second equivalence comes from what we just showed in (1), the third equivalence is from~\cite[Prop.~2.3]{emmanouil-kaperonis-K-flatness-pure}, and the last equivalence follows from statement (2).
\end{proof}

\section*{Acknowledgements}
This paper was advanced during a visit of the second author to the  Univeridad de Murcia. He thanks his coauthors and his host, Sergio Estrada, for his gracious hospitality.


\begin{thebibliography}{9}

\bibitem[AR94]{AR} J. Adamek and J. Rosicky.  Locally presentable and accessible categories.
London Mathematical Society Lecture Note Series, 189. Cambridge University Press, Cambridge, 1994.

\bibitem[BCE20]{bazzoni-cortes-estrada} Silvana Bazzoni, Manuel Cort\'es-Izurdiaga, and Sergio Estrada, \emph{Periodic modules and acyclic complexes}, Algebr. Represent. Theory vol.~23, no.~5, 2020, pp.~1861--1883.

 \bibitem[B\"uh10]{buhler-exact categories}
T.~B\"uhler, \emph{Exact Categories}, Expo. Math. vol.~28, no.~1, 2010, pp.~1--69.

\bibitem[CET21]{cet-G-flat-stable-scheme}
Lars Winther Christensen, Sergio Estrada, and Peder Thompson, \emph{The stable category of Gorenstein flat sheaves on a Noetherian scheme}, Proc. Amer. Math. Soc. vol.~149, no.~2, 2021, pp.~525--538.


\bibitem[Emm16]{emmanouil-pure-acyclic-complexes}
 Ioannis Emmanouil, \emph{On pure acyclic complexes},
J. Algebra vol. 465, 2016, pp.~190--213.

 \bibitem[Emm19]{emmanouil-relation-K-flatness-K-projectivity}
 Ioannis Emmanouil, \emph{On the relation between K-flatness and K-projectivity}, J. Algebra vol. 517, 2019, pp.~320--335.

\bibitem[Emm22]{emmanouil-K-flatness-and-orthogonality-in-homotopy-cats}
 Ioannis Emmanouil, \emph{K-flatness and orthogonality in homotopy categories},  Isr. J. Math. (2022).
 https://doi.org/10.1007/s11856-022-2413-4

\bibitem[EK22]{emmanouil-kaperonis-K-flatness-pure}
 Ioannis Emmanouil and Ilias Kaperonis, \emph{On K-absolutely pure complexes}, 2022 preprint at
 http://users.uoa.gr/~emmanoui/research.html.

\bibitem[EGO17]{estrada-gillespie-odabasi}
Sergio Estrada, James Gillespie, and Sinem Odabaşı, \emph{Pure exact structures and the pure derived category of a scheme}, Math. Proc. Cambridge Philos. Soc. vol.~163, no.~2, 2017, pp.~251--264.


\bibitem[EGO20]{estrada-guil-odabasi-phantom}
Sergio Estrada,  Pedro A. Guil Asensio, and Sinem Odabaşı, \emph{Phantom covering ideals in categories without enough projective morphisms}, Journal of Algebra vol.~562, 2020, pp.~94--114.

 \bibitem[Gil04]{gillespie}
 James Gillespie, \emph{The flat model structure on Ch(R)}, Trans.
 Amer. Math. Soc. vol.~356, no.~8, 2004, pp.~3369--3390.


  \bibitem[Gil08]{gillespie-degreewise-model-strucs}
 James Gillespie, \emph{Cotorsion pairs and degreewise homological model structures},
 Homology, Homotopy Appl. vol.~10, no.~1, 2008, pp.~283--304.


 \bibitem[Gil16a]{gillespie-G-derived}
 James Gillespie, \emph{The derived category with respect to a generator}, Ann. Mat. Pura Appl. (4) vol.~195, no.~2, 2016, pp.~371--402.

\bibitem[Gil16b]{gillespie-recollement}
James Gillespie, \emph{Gorenstein complexes and recollements from cotorsion pairs}, Adv. Math. vol.~291, 2016, pp.~859--911.

\bibitem[Gil17]{gillespie-ding-modules}
James Gillespie, \emph{On Ding injective, Ding projective and Ding flat modules and complexes},
Rocky Mountain J. Math., vol.~47, no.~8, 2017, pp.~2641--2673.

 \bibitem[Gil23a]{gillespie-K-flat}
  James Gillespie, \emph{K-flat complexes and derived categories}, Bull. Lond. Math. Soc. vol.~55, no. 1, 2023,  pp.~119--136.

\bibitem[Gil23b]{gillespie-ac-pure-proj}
James Gillespie, \emph{The homotopy category of acyclic complexes of pure-projective modules},  Forum Mathematicum vol.~35, no.~2, 2023, pp.~507--521. https://doi.org/10.1515/forum-2022-0183


 \bibitem[Har77]{hartshorne}
 Robin Hartshorne, \emph{Algebraic Geometry}, Grauate Texts in
 Mathematics vol.~52, Springer-Verlag, New York, 1977.

  \bibitem[Hov02]{hovey}
 Mark Hovey, \emph{Cotorsion pairs, model category structures,
 and representation theory}, Mathematische Zeitschrift, vol.~241,
 2002, pp.553--592.


\bibitem[Mur07]{murfet-thesis} Daniel Murfet, \emph{The mock homotopy category of projectives and Grothendieck duality}, PhD thesis, Australian National University, 2007. (online at www.therisingsea.org.)

\bibitem[PS23]{posit-stov-flat-sheaves}
Leonid Positselski and Jan \v{S}t'ov\'{\i}\v{c}ek, \emph{Flat quasi-coherent sheaves as direct limits, and quasi-coherent cotorsion periodicity}, arXiv:2212.09639v1.

\bibitem[Qui67]{quillen-model categoires}
D.~Quillen, \emph{Homotopical algebra}, SLNM vol.~43, Springer-Verlag, 1967.



\bibitem[S{\v{S}}11]{saorin-stovicek}
Manuel Saor{\'{\i}}n and Jan \v{S}t'ov\'{\i}\v{c}ek, \emph{On exact categories and applications to triangulated adjoints and model structures}, Adv. Math. vol.~228, no.~2, 2011, pp.~968--1007.

 \bibitem[Spa88]{spaltenstein} N.~Spaltenstein, \emph{Resolutions
 of unbounded complexes}, Compos. Math., vol.~65, no.~2, 1988,
 pp.~121-154.

 \bibitem[Wei94]{weibel}
 Charles A. Weibel, \emph{An Introduction to Homological Algebra},
 Cambridge Studies in Advanced Mathematics vol.~38, Cambridge
 University Press, 1994.
 \end{thebibliography}
\end{document}